\documentclass[12pt,a4paper]{amsart}

\usepackage{hyperref}

\usepackage{soul} 

\usepackage{mathabx}
\usepackage{xcolor}
\usepackage[latin1]{inputenc}
\usepackage[spanish,english]{babel}
\usepackage{graphicx}

\usepackage{tabularx}

\usepackage{amssymb}
\usepackage{amsmath}

\usepackage{mathrsfs}

\usepackage{enumitem}
\usepackage{pifont}

\usepackage{amsfonts}
\usepackage{graphicx}

\newtheorem{teo}{Theorem}[section]
\newtheorem{exer}[teo]{\sffamily\bfseries Ejercicio}
\newtheorem{lem}[teo]{Lemma}
\newtheorem{cor}[teo]{Corollary}
\newtheorem{prop}[teo]{Proposition}

\newtheorem{rem}[teo]{Remark}

\numberwithin{equation}{section}

\newcommand{\diag}{\operatorname{Diag}}
\newcommand{\kah}{\mathcal{K}(\mathcal{H})^{ah}}
\newcommand{\dbh}{\mathcal{D}\left(\mathcal{B}\left(\mathcal{H}\right)\right)}
\newcommand{\dkh}{\mathcal{D}\left(\mathcal{K}\left(\mathcal{H}\right)\right)}
\newcommand{\dah}{\mathcal{D}\left(\mathcal{B}\left(\mathcal{H}\right)\right)^{ah}}
\newcommand{\kh}{\mathcal{K}(\mathcal{H})}
\newcommand{\uh}{\mathcal{U}(\mathcal{H})}
\newcommand{\khh}{\mathcal{K}(\mathcal{H})^{h}}
\newcommand{\bh}{\mathcal{B}(\mathcal{H})}

\newcommand{\bah}{\mathcal{B}(\mathcal{H})^{ah}}
\newcommand{\ob}{\mathcal{O}_b}

\newcommand{\oo}{\mathcal{O}}
\newcommand{\uu}{\mathcal{U}}
\newcommand{\ua}{\mathcal{U}_{\aaa}}
\newcommand{\ukd}{{\mathcal {U}}_{k,d}}
\newcommand{\ub}{\mathcal{U}_{\bb}}

\newcommand{\uk}{\mathcal{U}_k}
\newcommand{\aaa}{\mathcal{K}+\C}

\newcommand{\bb}{\D(\mathcal{K}+\C)}
\newcommand{\kk}{\mathcal{K}}
\newcommand{\dd}{\mathcal{D}}

\newcommand{\PP}{\mathcal{P}}

\newcommand{\h}{\mathcal{H}}

\newcommand{\N}{\mathbb N}
\newcommand{\C}{\mathbb C}

\newcommand{\R}{\mathbb R}

\newcommand{\I}{\mathcal I}

\newcommand{\longi}{{\rm L}}
\newcommand{\dist}{{\rm dist}}
\newcommand{\bit}{\begin{itemize}}
	\newcommand{\eit}{\end{itemize}}
\newcommand{\be}{\begin{enumerate}}
	\newcommand{\ee}{\end{enumerate}}
\newcommand{\bx}[1]{\begin{exer}\rm{#1}}
	\newcommand{\ex}{\end{exer}}

\newcommand{\ba}{\begin{array}}
	\newcommand{\ea}{\end{array}}
\newcommand{\bc}{\begin{center}}
	\newcommand{\ec}{\end{center}}

\newcommand{\g}{\gamma}

\newcommand{\D}{\mathcal{D}}

\newcommand{\bq}{\begin{equation}}
\newcommand{\eq}{\end{equation}}

\begin{document}

\title{Geodesic neighborhoods in unitary orbits of self-adjoint operators of  
	$\mathcal{K} +  \C $
}

\author{Tamara Bottazzi $^1$ and Alejandro Varela$^{2,3}$}
\date{\today}

\address{$^1$ Sede Andina, Universidad Nacional de R\'io Negro, (8400) S.C. de Bariloche, Argentina}

\address{$^2$Instituto Argentino de Matem\'atica ``Alberto P. Calder\'on", Saavedra 15 3er. piso, (C1083ACA) Buenos Aires, Argentina}

\address{$^3$Instituto de Ciencias, Universidad Nacional de Gral. Sarmiento, J.	
	M. Gutierrez 1150, (B1613GSX) Los Polvorines, Argentina}

\email{tbottazzi@unrn.edu.ar, avarela@campus.ungs.edu.ar }


\maketitle 
\begin{abstract} 	
In the present paper, we study the unitary orbit of a compact Hermitian 
diagonal operator with spectral multiplicity one under the action of the unitary group $\ua$ of the unitization of the compact operators $\kh+\C$
, or equivalently, the quotient $\ua/\ub$. 
We relate this and the action of different unitary subgroups to describe metric geodesics (using a natural distance) which join end points. 
As a consequence we obtain a local Hopf-Rinow theorem.
We also explore cases about the uniqueness of short curves and prove that there exist some of these that cannot be parameterized using minimal anti-Hermitian operators of $\kh+\C$.
\end{abstract}

\section{Introduction}

Let $\mathcal{A}$ be a C$^*$-algebra, $\mathcal{B}$ a von Neumann subalgebra of $\mathcal{A}$, and $\mathcal{U}_\mathcal{A}$, $\mathcal{U}_\mathcal{B}$ its respective unitary groups. 
Theorem II of \cite{dmr1} or Theorem I-2 of \cite{dmr2} (and 
the Remark that follows it) imply that for every element 
$\rho\in\uu_{\mathcal{A}}/\uu_{\mathcal{B}}$, and every tangent vector $x\in 
T_\rho\left(\uu_{\mathcal{A}}/\uu_{\mathcal{B}}\right)$, there exist a minimal 
lift $Z\in \mathcal{A}^{ah}$ of $x$ ($x=Z\rho-\rho Z$ and $\|Z\|\leq \|Z+D\|$ for all 
$D\in \mathcal{B}$) such that the curves
\begin{equation}\label{def: curva gama con Z y rho}
\gamma(t)=e^{t Z} \rho e^{-t Z} 
\end{equation}
are all the possible short curves (under a natural distance) starting at $\rho$ with fixed initial velocity $x$. In this context, we will call this $Z$ a minimal operator.

Moreover in \cite{dmr2} local and global Hopf-Rinow theorems were proved in 
this context with additional hypothesis on the unitary groups involved.

If the assumption of $\mathcal{B}$ being a von Neumann algebra is relaxed, 
Theorem I of \cite{dmr1} proves that every minimal lift $Z$ still produces a short 
curve if $\mathcal{B}$ is only required to be a C$^*$-algebra. 
Nevertheless, in this case such a $Z$ may not exist 
for every tangent vector $x$ (see for example the discussion following 
Proposition 18 in \cite{bottazzi_varela_LAA} or the properties of $Z_2$ defined 
in \eqref{defi Z2} in this paper). Therefore, if $\mathcal{A}$ and 
$\mathcal{B}$ are $C^*$-algebras the parameterization of minimal curves with 
arbitrary initial velocity is not known in general nor the existence 
of geodesic neighborhoods.
The main objective of this work is the study of short curves of a particular example where the subalgebra $\mathcal{B}$ is not a von Neumann algebra.

Denote with $\kk+\C$ the C$^*$-algebra obtained after the unitization of the compact operators in $\bh$, that is $\{X\in\bh: X=K+\theta I, K\in\kk,\ \theta\in\C\}$. $\h$ will be a separable Hilbert space, and $\ua$ will denote the unitary operators of $\kh+\C$.
If we fix an orthonormal basis $\{e_i\}_{i\in\N}$  in $\h$, we can consider matricial characterizations in $\bh$ an diagonal operators. Let $b$ be a compact diagonal self-adjoint operator with spectral multiplicity one, and $\ob$ the orbit
\begin{equation} 
\ob=\ob^{\ua}=\{ubu^*:u\in\ua\}
\end{equation}
This orbit has a structure of a smooth homogeneous space under the action of 
$\ua$ with the identification $\ob \simeq \ua/\ub$ (see for example Lemma 1 of 
\cite{bottazzi_varela_DGA} and the discussion that follows it).

As we comment in Remark \ref{rem: orbitas y metrica finlser iguales}  the 
homogeneous space $\ob$ coincides with the orbit of $b$ under the 
action of several other unitary subgroups. Moreover, a natural Finsler metric 
defined on the tangent spaces (see \ref{rem: espacio tangente Ob}) and a 
distance (see \ref{def: distancia rectificable}) in $\ob$ is also preserved if 
we consider those different unitary subgroups.

%
%

In this context, we analyze geodesic neighborhoods of $\ob$ and cases of short curves satisfying initial conditions or connecting given endpoints that cannot obtained using minimal operators $V\in \left(\kh+\C\right)^{ah}$.
 In Theorem \ref{teo minimal} it is shown that short curves in $\ob$
of the form \eqref{def: curva gama con Z y rho} can be constructed using minimal operators $Z\in \left(\kh+ \dd(\bh) \right)^{ah}$.
These geodesics and a result from \cite{dmr1} allows us to prove in \ref{teo: hopf rinow local} a Hopf-Rinow local theorem for $\ob$.
We also consider certain types of minimal operators with diagonal belonging to $\dd(\bh)\setminus \dd(\kh+\C)$ and construct with them short curves $\g$ in $\ob$ (for the distance \eqref{def: distancia rectificable}) such that in a fixed neighborhood, there is not any curve $\delta$ of the form $\delta(t)=e^{tV}be^{-tV}$ with $V$ a minimal vector in $(\kk+\C)^{ah}$  that starts in $b$ and ends in $\g(t)$ (see \ref{teo: unicidad curva minimal}). This means that there exist geodesics that cannot be obtained using minimal vectors $V$ in $(\kk+\C)^{ah}$.

In the general context mentioned at the beginning of this section the previous results imply that if $\mathcal{B}$ is only required to be a C$^*$-algebra then, even when Hopf-Rinow type theorems can be obtained, there exist short curves in $\mathcal{U}_\mathcal{A}/\mathcal{U}_\mathcal{B}$ that cannot be described using minimal elements of $\mathcal{A}$.

\section{Preliminaries}
Let $\bh$ be the algebra of bounded operators on a separable Hilbert space 
$\h$, and $\kk(\h)$ and $\uh$ the compact and unitary operators respectively. 
If an orthonormal basis $\{e_i\}_{i\in\N}$ is fixed we can consider matricial 
representations of each $A\in\bh$, that is 
$A=(A_{i,j})_{i,j\in\N}=(\left\langle Ae_i,e_j \right\rangle)_{i,j\in\N}$
 and diagonal operators which we denote with $\dbh$. Any $D\in\dbh$ fulfills $\left\langle De_i,e_j \right\rangle=0 $ for every $i\neq j $. 
 
 With the preceding notation, we define 
 columns (and similarly, rows) of any 
 operator $A\in\bh$ as 
 $c_j(A)=\sum_{i=1}^{\infty}\left\langle 
 Ae_i,e_j 
 \right\rangle e_j=\left(A_{1j},A_{2j},... 
 \right) \subset \ell^2$, for each 
 $j\in\N$. 

We call $\uk$ the Fredholm subgroup of $\uu(\h)$, defined as
\begin{equation}\label{def de Uk}
	\begin{split}
	\uk&=\{u\in \mathcal{U}(\h): u-I\in \kk(\h)\}\\
	&=\{u\in\uh: \exists\  K\in\kah, u=e^K\},
	\end{split}
\end{equation}	
and the subgroup studied in \cite{bottazzi_varela_studia}:
\begin{equation}\label{definicion de Ukd}
\begin{split}
\ukd&=\{u\in \mathcal{U}(\h): u-e^{D}\in \kk(\h) \text{ for } D\in \left(\dbh\right)^{ah}  \}\\
&=\{u\in\uh: \exists\ K\in\kah \text{ and } \left(\dbh\right)^{ah}, \text{ such that } u=e^{K}e^D\},
\end{split}
\end{equation}
where $I$ is the identity in $\bh$ and the superscript $^{ah}$ means 
anti-Hermitian as well as $^h$ means 
Hermitian.

Consider the unitization of $\kh$
$$\aaa=\kh+\{\lambda I:\ \lambda \in \C\}\subset \bh,$$
endowed with the norm 
$$
||K+\lambda I\|_{\aaa}= 
\sup\{\|K C+\lambda C\|: C\in \kh, \|C\|=1\},
$$
for any $K+\lambda I\in \aaa$ (here $\|\cdot \|$ is the usual operator norm in $\bh$). The $\|\ \|_{\aaa}$ norm coincides with the operator norm in $\bh$:
$$
||K+\lambda I||_{\aaa}= \|K+\lambda I\|
$$
(this follows after considering the multiplication $(K+\lambda I)C$ for $C=h\otimes h\in\kh$, with $h\in H$ and $\|h\|=1$).

The space $(\aaa,||\cdot||_{\aaa})$ is a unital $C^*$-algebra with unit 
$I_{\aaa}=0+1.I=I$.

Let $\dkh= \dbh \cap \kh $ and define the subspace of diagonal operators of $\aaa$, given by 
$$
\bb=\dkh+\{\lambda I:\ \lambda \in \C\}.
$$
It is apparent that $\bb$ is a unital $C^*$-subalgebra of $\aaa$ and $I_{\aaa}\in \bb $.

If $u=K+\lambda I\in (\aaa)$ is a unitary operator, direct computations show 
that $KK^*=K^*K$ and $|\lambda|=1$. Therefore, there exists $\theta\in\R$ 
($\lambda=e^{i \theta}$) such that $u$ verifies that $u-e^{i\theta I}\in\kh$. 
Then $u\in \ukd$ (see \eqref{definicion de Ukd}) and therefore there exists $ 
K_0\in \kah $ such that $u=e^{K_0} e^{i\theta I}$ for the same $\theta$ (as 
seen in the proof of Proposition 3.3 in  \cite{bottazzi_varela_studia}). 
Moreover, it is apparent that if $u=e^{K} e^{i\theta I}$, with $\theta\in\R$ 
and $K\in\kah$, then $u=  e^{i\theta I} +\left(\sum_{n\geq 1} \frac{K^n}{n!} 
e^{i\theta I} \right)  \in \ua$, the unitary group of $\aaa$.

Similar considerations can be made with the unitaries of $\bb$. If $v\in \ub$ then $v=d+e^{i\theta I}$ with $d\in\dkh$ and $\theta\in\R$. This implies that $|d_{j,j}+e^{i\theta}|=1$ for all $j\in\N$ and therefore $ (d+e^{i\theta I})=e^{iR} $ with $R$ a real diagonal matrix such that $R_{j,j}\to \theta$ when $j\to\infty$. Conversely, if $v=e^{i R}$ with $R_{j,j}\in\R$ and $\lim_{j\to \infty} R_{j,j}=\theta$, then $v=e^{i(R-\theta I)} e^{i\theta I}\in \ub$ because $\lim_{j\to\infty}(R_{j,j}-\theta )=0$ and therefore $i(R-\theta I)\in\kah$.

Then the unitary groups of $\aaa$ and $\bb$ can be described as follows:
\begin{equation}\label{def de UA}
\begin{split}
\ua&=\uk.\{e^{i\theta}I:\ \theta\in \R\}=\{e^Ke^{i\theta I}: K\in\kah \text{ and } \theta \in \R \}\\
&=\{e^{K+i\theta I}:\ K\in \kah\ \text{and}\ \theta\in \R\}
\end{split}
\end{equation} 
and
\begin{equation*}
\begin{split}
\ub&=\{ d+e^{i\theta I}: d\in \dkh , \ \theta\in \R \text{ such that } |d_{j,j}+ e^{i\theta I}|=1 \}\\
&=\{ e^{d_0+i \theta I} :\ d_0\in\dkh^{ah} \text{ and } \theta\in\R\}\\
&=\{ e^{L_0} :\ L_0\in\dbh^{ah} \text{ such that } \lim_{j\to\infty} (L_0)_{j,j}=i \theta \text{ with } \theta\in\R \}.
\end{split}
\end{equation*}
 
\section{The homogeneous unitary orbit of a self-adjoint 
compact operator}
 Given a subgroup $\mathcal{U}\subset \mathcal{U}(\h)$ we will denote with 
 $\oo_b^{\mathcal{U}}$ the orbit of self-adjoint element $b\in\kh^h$ by a 
 subgroup $\mathcal{U}\subset \uh$, that is 
 $$
 \oo_b^{\mathcal{U}}=\{ubu^*:u\in\mathcal{U}\}
 $$
 Let $b=\diag\left(\{b_i\}_{i\in \N}\right)\in \D(\kh^h)$ denote the diagonal 
 operator with the sequence $\{b_i\}_{i\in \N}$ in its diagonal. We study the  
 unitary orbit of $b\in\kh\subset\aaa$ with $b_i\neq b_j$ for each $i\neq j$  
 under the action of $\ua$:
\begin{equation} \label{def: orbita ob}
\ob=\ob^{\ua}=\{ubu^*:u\in\ua\}.
\end{equation}
Observe that it is apparent that the following inclusions hold for these subgroups of $\uh$:
$$
\uk\subsetneq \ua\subsetneq\ukd.
$$
Nevertheless the orbit of $b$ under the three subgroups is the same set $\ob$ because 
\begin{equation}
\label{eq: las orbitas son iguales con distintos subgrupos}
e^{K }be^{-K }=e^{K}e^{itI}be^{-itI}e^{-K}=e^{K+itI}be^{-K-itI}
\end{equation}
for $t\in\R$ and then $\oo_b^{\uk}=\oo_b^{\ukd}$ (see for example Remark 4.4 in 
\cite{bottazzi_varela_studia}). Moreover, we will show further that the three of them share the same natural Finsler metric on the tangent spaces (as seen in Remark 4.5 in 
\cite{bottazzi_varela_studia} for the cases of $\ob^{\uk}$ and $\ob^{\ukd}$).

 $\ob$ has a smooth structure as described in Lemma 1 in \cite{bottazzi_varela_DGA} and the comments that follow it.  


The isotropy at any $c\in \ob=\ob^{\ua}$ is $\I_c=\{u\in \ua:\ ucu^*=c\}$. In particular, 
$$
\I_b=\left\{u\in \ua:\ [u,b]=0 \right\}=\ub.
$$
\begin{rem} \label{rem: espacio tangente Ob}
If $c\in \ob$, the following identification can be made:
$$
T_c\ob \cong (T\ua)_1/(T\I_b)_1 
=(\kk+\C)^{ah}/(\dd(\kk + 
\C))^{ah}. 
$$
\end{rem}
Observe that 
$$(\aaa)^{ah}/\bb^{ah}=\left\lbrace [X]:\ X=K_0+i\theta_0 I,\ K_0\in \kah\ \text{and}\ \theta_0\in \R  \right\rbrace. $$

where $[X]$ is the class defined as $Y\in[X]$ iff $Y=X+d+i\theta I$ for $d\in \D(\kah)$ and $\theta\in \R$. This 
 quotient space is endowed with the usual quotient norm, that in this case for 
 $X=K_0+i\theta_0 I$ is 
\begin{equation*}
\begin{split}
\|[X]\|=\|[K_0+i\theta_0 I]\|&= \inf_{\theta\in\R;\ d\in \D(\kah)} \|K_0+i\theta_0 I+d+i\theta I\|_{\aaa}
\\
&=\inf_{\theta\in\R;\ d\in \D(\kah)} \|K_0+d+i\theta I\|_{\aaa}.
\end{split}
\end{equation*}
In this context, a natural Finsler metric for any $x\in T_b\ob$, $x=Xb-bX$, with $X\in(\aaa)^{ah}$ is
\begin{equation}\label{eq: def metrica de Finsler}
\begin{split}
\|x\|_b&=\inf\{\|Y\|:\ Y\in (\aaa)^{ah},\ Yb-bY=Xb-bX\}\\
&=\mathop{\inf_{d\in \D(\kah)}}_{\theta\in\R} \|X+d+i\theta I\|_{\aaa} =\mathop{\inf_{d\in \D(\kah)}}_{\theta\in\R} \|X+d+i\theta I\|
\end{split}
\end{equation}
where $X+d+i\theta I$ is any element of the class $[X]=\{Y\in(\aaa)^{ah}: 
Y=X+d+i\theta I, \text{ for } \theta\in\R\text{ and } d\in \D(\kah)\} $ in 
$(\aaa)^{ah}/\bb^{ah}$.
 
An element $Y\in \bah$ such that $Yb-bY=x$ and 
$\left\|Y\right\|_{\aaa}=\|Y\|=\|[X]\|=\left\|x\right\|_b$ will be called a 
minimal lifting for $x$, and its diagonal will be a minimal diagonal 
approximant (or minimizing diagonal) for $Y$. This operator $Y$ may not be 
compact nor unique (see 
\cite{bottazzi_varela_LAA}), and it will be called a minimal operator.

Given any $c=e^{K+itI}be^{-K-itI}=e^{K }be^{-K }\in\oo_b$ ($K\in\kah$, 
$t\in\R$) we can define the norm in its tangent space $T_c \oo_b$ using that 
$z=Zc-cZ\in T_c\oo_b$ for $Z=e^KXe^{-K}\in\kah$ and $Xb-bX\in T_b \oo_b$. Then 
the Finsler norm in $T_c \oo_b$ is  $\|z\|_c =\|[Z]\|=\inf\{\|Y\|:\ Y\in 
(\aaa)^{ah},\ Yc-cY=Zc-cZ\}=\|[X]\|$.
Note that this norm is invariant under the action of $\ua$. 
 
\begin{rem}\label{rem: orbitas y metrica finlser iguales} 
As it was mentioned before in \eqref{eq: las orbitas son iguales con distintos subgrupos}, the following orbits are equal
	$$
	 \oo_b^{\ua}=\oo_b^{\uk}=\oo_b^{\ukd} 
	$$
for $\uk\subsetneq \ua \subsetneq \ukd$ defined in \eqref{def de Uk}, \eqref{def de UA} and \eqref{definicion de Ukd} respectively (see Proposition 4.1, 4.4 and Remark 4.7 in \cite{bottazzi_varela_studia}). 
 
Let $X\in(\aaa)^{ah}$, $X=K+it I$ with $K\in\kah$ and $t\in\R$, then 
\begin{equation}\label{eq: igualdades entre infimos para calculo normas 
cociente}
\begin{split}
\inf_{D\in \D(\bah)} \|K+D\|&=\inf_{D\in \D(\bah)} \|\overset{X}{\overbrace{K+itI}}+D\|\leq
\mathop{\inf_{d\in \D(\kah)}}_{\theta\in\R} \|X+i\theta I+d\|\\ 
&
=\mathop{\inf_{d\in \D(\kah)}}_{\theta\in\R} \|K+i\theta I+d\|
\leq 
 \inf_{d\in \D(\kah)} \|K+d\|.
\end{split}
\end{equation}
But since  $\displaystyle
\inf_{d\in \D(\kah)} \|K+d\| =\inf_{D\in \D(\bah)} \|K+D\|
$, [Prop. 5,\cite{bottazzi_varela_LAA}], the previous inequalities imply that all those infimums are equal.

This means that also the Finsler metric defined for $T_c \ob$  in \eqref{eq: 
def metrica de Finsler} coincides if we consider any of the quotients 
$\kh^{ah}/ \mathcal{D}(\kah)$, $(\aaa)^{ah}/\bb^{ah}$ or $(\kah +\dah) /\dah$ 
(see also Remark 4.5 in \cite{bottazzi_varela_studia}).
\end{rem}
\begin{rem}
	Observe that nevertheless $\ob\subsetneq \ob^{\uh}$. This follows because 
	if we suppose that for every $X\in\bah$ holds that $e^X b e^{-X}=e^K b 
	e^{-K}$ for $K\in \kah$, then $e^{-K}e^X$ must be a diagonal 
	unitary. Therefore, for all $X\in\bah$ we could write that $e^X=e^K e^D$ 
	for $D\in\dah$ which is known to be false 
	(see for example Remark 3.7 in \cite{bottazzi_varela_studia}).
\end{rem}
Consider piecewise smooth curves $\beta:[a,b]\to \ob^{\ua}$. We define 
\begin{equation}\label{def: longitud rectificable}
\longi(\beta)=\int_a^b\left\|\beta'(t)\right\|_{\beta(t)}\ dt\ 
\text{ , and }
\end{equation}
\begin{equation}\label{def: distancia rectificable}
\dist(c_1,c_2)=\inf\left\{\longi(\beta):\beta\ {\rm is\ smooth},\beta(a)=c_1,\beta(b)=c_2\right\} 
\end{equation}
as the rectifiable length of $\beta$ and distance between two points $c_1, c_2\in\ob^{\ua}$, respectively.

%
	
\section{A short curve in $\ob$ obtained acting with a minimal operator not 
belonging to $ 
(\aaa)^{ah}/\bb^{ah}$}
Consider the following operator described as an infinite matrix:
\begin{equation*}
Z_{\delta,\gamma}=i
\left(
\begin{array}{ccccccccccccc}
0 & -\delta   & \gamma    & -\delta ^2  & \gamma ^2  & -\delta ^3  & \gamma ^3  &\cdots \\
-\delta  & 0 & \gamma  & -\delta ^2 & \gamma ^2 & -\delta ^3 & \gamma ^3 &\cdots  \\
\gamma   & \gamma  & 0 & -\delta ^2 & \gamma ^2 & -\delta ^3 & \gamma ^3 &\cdots  \\
-\delta ^2  & -\delta ^2 & -\delta ^2 & 0 & \gamma ^2 & -\delta ^3 & \gamma ^3 &\cdots \\
\gamma ^2  & \gamma ^2 & \gamma ^2 & \gamma ^2 & 0 & -\delta ^3 & \gamma ^3 &\cdots  \\
-\delta ^3  & -\delta ^3 & -\delta ^3 & -\delta ^3 & -\delta ^3 & 0 & \gamma ^3&\cdots\\
\gamma ^3  & \gamma ^3 & \gamma ^3 & \gamma ^3 & \gamma ^3 & \gamma ^3 & 0 &\cdots \\
\vdots&\vdots&\vdots&\vdots&\vdots&\vdots&\vdots&\ddots
\end{array}
\right),\ \text{with}\ \g, \delta\in (0,1) .
\end{equation*}
$Z_{\delta,\gamma}$ is a Hilbert-Schmidt operator which has been studied in 
\cite{bottazzi_varela_DGA} in its 
self-adjoint version. We recall here some 
of 
its 
properties.

Let $\gamma^2= \delta$ and $\delta^2< \gamma$ (for example $\gamma=1/2$ and 
$\delta=1/4$), and denote with $Z_{\delta,\gamma}^{[1]}$ the operator defined 
by the matrix of $Z_{\delta,\gamma}$ with zeros in the first column and row, 
with $c_1(Z_{\delta,\gamma})$ the first column of $Z_{\delta,\gamma}$, and with 
$D_0$ the (uniquely determined) diagonal matrix such that 
every  
$(D_0)_{i,i}$ is chosen to satisfy $c_i(Z_{\delta,\gamma})\perp 
c_1(Z_{\delta,\gamma})$ for all $i\neq 1$. Then 
\begin{equation}\label{defi Zo}
Z_o= \frac{\|Z_{\delta,\gamma}^{[1]}+D_0\|}{\|c_1(Z_{\delta,\gamma})\|} (Z_{\delta,\gamma}-Z_{\delta,\gamma}^{[1]})+Z_{\delta,\gamma}^{[1]}
\end{equation}
is also a Hilbert-Schmidt operator with the property that $D_0$ (constructed as 
mentioned before) is a minimal diagonal for $Z_o$ (see (5.1) in 
\cite{bottazzi_varela_DGA} for a detailed proof of these statements  
and the following comments).
Moreover, it has been proved that $D_0\in \D(\bah)$ is the unique bounded best 
approximant anti-Hermitian diagonal of $Z_o$. $D_0$ has the particular property 
that $\lim_{k\to\infty}(D_0)_{2k,2k}\neq \lim_{k\to\infty}(D_0)_{2k+1,2k+1}$ 
and both limits are not null. Therefore $D_0$ is not compact and we call it an 
oscillant diagonal. We will write with
\begin{equation} \label{defi Z2}
Z_2=Z_o+D_0
\end{equation}
to denote the minimal operator constructed as above.

Then
\begin{equation}
\label{ecuacion que cumple D0 minimal de Zo}
\dist(Z_o,\D(\kah))=\|[Z_o]\|_{\kah/\D(\kah)}
=\|Z_2\|=\|c_1(Z_2)\|=\|c_1(Z_o)\|.
\end{equation}

\begin{teo} \label{teo minimal}
Let $b=\diag\left(\{b_i\}_{i\in \N}\right)\in \D(\kh^h)$ with  $b_i\neq b_j$ for each $i\neq j$. Consider the unitary orbit
$\ob^{\ua}$ defined in \eqref{def: orbita 
ob} and $x=Z b-bZ\in  T _b\ob $, for 
$Z$ a minimal operator in $\kah+\dah$. Then the uniparametric group curve 
$\g(t)= 
e^{tZ}be^{-tZ}$ has minimal length in the class of all curves 
in $\ob^{\ua}$ joining $\g(0)$ and $\g(t)$ for each 
$t\in\left[-\frac{\pi}{2\|Z \|},\frac{\pi}{2\|Z \|}
 \right] $.
\end{teo}
\begin{proof} This proof is a direct consequence of mentioned previous results, 
but we include here the citations and reasonings for the sake of clarity. 
	
By Theorem 4.2 in \cite{bottazzi_varela_studia}, $\g(t)\in \ob^{\uk}$ for any $t\in \R$. Using Remark \ref{rem: orbitas y metrica finlser iguales}, we obtain that $\g(t)\in \ob^{\ua}=\ob,  \forall\ t$. Moreover,
$$\|x\|_b=\|Z b-bZ \|_b=\|[Z ]\|=\inf_{\theta\in\R;\ d\in \D(\kah)} \|Z 
+d+i\theta I\|=\|Z \|,$$
where the minimality of $Z $ implies the last equality.

Consider $\PP_b=\{ubu^*:\ u\in\uu(\h)\}$, then by Theorem II in \cite{dmr1}, 
since $Z $ is minimal, the curve $\gamma$ has minimal length over all the 
smooth curves in $\PP_b$ that join $\gamma(0)=b$ and $\gamma(t)$, with 
$\left|t\right|\leq \frac{\pi}{2\left\|Z \right\|}$. Since clearly 
$\ob^{\ua}\subseteq\PP_b$, then for each $t_0\in 
\left[-\frac{\pi}{2\left\|Z\right\|},\frac{\pi}{2\left\|Z\right\|}\right]$
 follows that $\gamma$ is a short curve in $\ob^{\ua}$, that is
$$
\longi\left(\gamma\big|_{[0,t_0]}\right)=\dist(b,\gamma(t_0)),$$
where $\dist(b,\gamma(t_0))$ is the rectifiable distance between $b$ and 
$\gamma(t_0)$ defined in \eqref{def: distancia rectificable}.
\end{proof}

\begin{cor} \label{coro Z2 minimal}
	Let $b=\diag\left(\{b_i\}_{i\in \N}\right)\in \D(\kh^h)$ with  
	$b_i\neq b_j$ for each $i\neq j$. Consider the unitary orbit
	$\ob^{\ua}$ defined in \eqref{def: orbita 
		ob}, $x=Z_ob-bZ_o\in  T _b\ob $, for 
	$Z_o$ defined in \eqref{defi Zo}, with $D_0$ its unique minimizing
	diagonal, and $Z_2=Z_o+D_0$ defined in \eqref{defi Z2}. 
	
	Then the 
	uniparametric group curve $\g(t)= 
	e^{tZ_2}be^{-tZ_2}$ has minimal length in the class of all 
	curves 	in $\ob^{\ua}$ joining $\g(0)$ and $\g(t)$ for each 
	$t\in\left[-\frac{\pi}{2\|Z_2\|},\frac{\pi}{2\|Z_2\|}
	\right] $.
\end{cor}
\begin{proof}
	If we consider $Z=Z_2=Z_o+D_0$ in the statements of Theorem \ref{teo 
	minimal}, then $Z_2$ satisfies the conditions required and therefore the 
	proof is apparent.
\end{proof}

The previous result will allow us to state that the converse 
of Theorem I in \cite{dmr1} does not 
necessary hold 
when the subalgebra considered (here 
$\dd(\kk+\C)$) is not a von 
Neumann algebra. Let us describe the 
context of that article. Let $\mathcal{A}$ 
be a C$^*$-algebra and $\mathcal{B}$ a 
C$^*$-subalgebra, then a natural Finsler metric as the one in \eqref{eq: def 
metrica de Finsler} is defined for the generalized flag $\mathcal{P}=\uu_{\mathcal{A}}/\uu_{\mathcal{B}}$.
If the element $X\in\mathcal{A}^{ah}$ is minimal for a tangent vector $x\in 
T_p(\uu_{\mathcal{A}}/\uu_{\mathcal{B}})\simeq 
T_1(\uu_{\mathcal{A}})/T_1(\uu_{\mathcal{B}})$ (that is: $x=Xp-pX$ and  
$\|X\|=\inf\{\|Y\|:Y\in\mathcal{A}^{ah} ; Yp-pY=x\}$) then the curve 
$\gamma(t)=L_{e^{tX}} \cdot p $ has minimal length for 
$|t|\leq\frac{\pi}{2\|X\|}$ (where $L$ is a left action on $\mathcal{P}$) with 
the distance defined in \eqref{def: distancia rectificable}.

The following result shows that there might exist some minimal curves $\gamma$ 
in these generalized flags $\mathcal P=\ua/\ub$ that are not of the form 
$\gamma(t)=L_{e^{tZ}}\cdot p$ for $Z\in\mathcal{A}^{ah}$ a minimal lifting of a 
tangent vector $x\in T_1(\uu_{\mathcal{A}})/T_1(\uu_{\mathcal{B}})$.

\begin{rem}\label{rem: curvas cortas en Ob con velocidades fuera de KmasC} 

	Let $Z_2$ be the operator defined in \eqref{defi Z2}. As it was mentioned 
	in Corollary \ref{coro Z2 minimal}, the uniparametric curve $\gamma(t)= 
	e^{tZ_2}be^{-tZ_2}$ has minimal length in the class of all 
	curves in $\ob=\ob^{\ua}$ joining $\g(0)$ and $\g(t)$ for each 
	$t\in\left[-\frac{\pi}{2\|Z_2\|},\frac{\pi}{2\|Z_2\|}
	 \right] $. 
	
Therefore the curve $\gamma$ is included in $\ob$ with initial conditions 
$\gamma(0)=b$, 
$\g'(0)=x=Z_ob-bZ_o$ even for velocity vectors $x\in T_b \ob $ that do not have 
a minimal compact lifting $K_0$ (recall that $Z_o+D_0$ is not compact and $D_0$ is its unique minimizing diagonal). Thus 
$\ua$ is an example of a group  whose 
action on $\ob$ has short curves that might not be described using minimal 
vectors 
$Y\in(\aaa)^{ah}$. This 
is not new (for instance, see Remark 4.7 in \cite{bottazzi_varela_studia}), but 
in the present case $\aaa$ and $\bb$, whose anti-Hermitian elements are the 
Lie-algebras of $\ua$ and $\ub$ respectively, are unital $C^*$-algebras.

We will develop some details of this situation in the next section.
\end{rem}

\section{Neighborhoods of short curves defined by minimal vectors in 
$\ua/\ub$}

In this section we will consider the problem of the existence of a neighborhood 
around $b\in \diag(\kh)^h$ with $b_{i,i}\neq b_{j,j}$ for $i\neq j$ whose 
elements can be joined with $b$ with a short curve of the form
$$
\gamma(t)=e^{tZ}be^{-tZ}
$$
for some minimal anti-Hermitian element $Z\in(\aaa)^{ah}$ and $t$ in some 
interval.

Recall here $Z_2=Z_o+D_0$ defined in \eqref{defi Z2} where $Z_o\in\kah$ is the 
Hilbert-Schmidt operator defined in \eqref{defi Zo} and $D_0$ is its unique 
minimizing diagonal with the property that $D_0$ has subsequences that converge 
to two different (not null) limits as described in the previous section.
Moreover, $Z_2$ satisfies, 
\begin{enumerate}
	\item $\|Z_2\|=\|c_{1}(Z_2)\|$, 
	\item $c_{1}(Z_2)_{1}=(Z_2)_{1,1} = 0$ and 	
	\item $c_{1}(Z_2)_{j}=(Z_2)_{j,1} \neq 0$ for all $j\neq 1$.
\end{enumerate}

\begin{lem}\label{lem: sobre geodesica con igual veloc inicial} 
	Let $\g(t)=e^{t Z } b e^{-t 
Z}\subset \ob$ with $Z \in \bah$ a minimal operator with unique minimizing 
diagonal, and consider a curve $\delta(t)=e^{t V} b e^{-tV}$, with $V \in \bah$ 
another minimal operator such that $\delta'(0)=Vb-bV=\gamma'(0)=Z
b-bZ$. Then it must be $Z=V$.
\end{lem}
\begin{proof} 
	Since $\delta'(0)=\g'(0)$, 
	then
	$$
	\delta'(0)=Vb-bV=\g'(0)=Z b-bZ
	$$
	and therefore $V-Z$ commutes with $b$. Then 
	$V-Z\in\mathcal{D}(\bah)$ which implies that $Z$ and $V$ must be 
	equal outside their diagonals. Then, since $\diag(Z)$ is the only 
	minimizing diagonal for $Z$ and $V$ is also a minimal operator then 
	$\diag(V)=\diag(Z)$ which implies that $V=Z$.
\end{proof}
\begin{rem}
	Observe that if we apply the 
	previous lemma 
	to the case where $Z=Z_2$ defined in \eqref{defi Z2} and $\delta$ and $V$ 
	satisfy 
	the assumptions of the lemma, then in particular $V\notin 
	(\kk+\C)^{ah}$. This is a direct consequence of the fact that $Z_2$ has two 
	different non zero diagonal limits, something that $V\in (\kk+\C)$ cannot 
	satisfy.
\end{rem}

\begin{lem}\label{lem: gamma y delta se cruzan 
entonces hay una diagonal...} 
	
	Let $Z=K_Z+\diag(Z)$ with $K_Z\in\kah$, 
	$\diag(Z)\in \dah$ 
	be a minimal operator and $\gamma: 
	[0,\frac{\pi}{2\|Z\|}] \to \ob$ the short 
	curve defined as $\gamma(t)=e^{tZ} b e^{-tZ}$ 
	(see  Theorem \ref{teo minimal}) 
	and let 
	$\delta:\left[0, \frac{\pi}{2\|V\|}\right]\to \ob$  be another short 
	curve defined by 
	$\delta(s)=e^{sV} b e^{-sV}$ for 
	$V=K_V+\diag(V)$ another minimal operator with 
	$K_V\in \kah$ and 
	$\diag(V)\in\dah$. 
	Moreover,  suppose that there exists $t_1
	\in \left(0,\frac{\log(2)/8}{\|Z\|}\right]$ and $s_1\in 
	\left(0, \frac{\pi}{2\|V\|} \right]$ such that 	
	$\gamma(t_1)=\delta(s_1)$.
	
	Then 
	$$
	e^{t_1Z}=e^{s_1V} 
	e^{-\diag(s_1V)+\diag(t_1Z)}\ \text{ and } 
	\|s_1V\|=\|t_1Z\|.
	$$ 
\end{lem}
\begin{proof} 
Note that $t_1$ satisfies $ \|t_1 Z\|<(\log 2)/8 $, and then $t_1Z$ is 
sufficiently close to 
	zero in the sense of Definition 2.1 of \cite{bottazzi_varela_studia}.
	
	Now consider the equality $\gamma(t_1)=\delta(s_1)$ 
	$$
	e^{s_1V} b e^{-s_1V}=e^{t_1Z} b e^{-t_1Z}.
	$$
	Then
	$$
	b e^{-s_1V}e^{t_1Z}=e^{-s_1V}e^{t_1Z} b 
	$$
	which implies that $b$ commutes with 
	$e^{-s_1V}e^{t_1Z}$.
	Therefore $e^{-s_1V}e^{t_1Z}$ is diagonal and 
	unitary. Then there exists $D\in\dbh^{ah}$ 
	such that $e^{-s_1V}e^{t_1Z}=e^D$. 
	
	Observe that since $Z$ is a minimal operator 
	the length of $\gamma$ restricted 
	to $[0,t_1]$ is $\|t_1Z\|$ and if $\delta$ is 
	a short curve then the length 
	of $\delta$ must coincide with $\|t_1Z\|$ (see Theorem 4.1 in \cite{dmr1}). 
	Since the length of $\delta$ 
	restricted to $[0,s_1]$ equals $\|s_1 V\|$
    because $V$ is a minimal 
	operator, then $\|t_1Z\|=\|s_1 V\|$.
	Also $t_1 Z$ is sufficiently close to zero 
	(thus 
	$s_1V$), so we can apply Proposition 3.11 and 
	Corollary 3.12 of 
	\cite{bottazzi_varela_studia} to obtain 
	\begin{equation}
	\label{ec el producto de e-V eZ es e a las 
	diags}
	\begin{split}
	e^D&=e^{-s_1V}e^{t_1Z}\\
	&=  
	e^{-s_1V+\diag(s_1V)-\diag(s_1V)}e^{t_1Z-\diag(t_1Z)+\diag(t_1Z)}\\
	&=e^{K -\diag(s_1V)+\diag(t_1Z)}
	\end{split}
	\end{equation}
	for $K\in\kah$ with $\diag(K)=0$.
	Then, since $\|t_1Z\|=\|s_1V\|<(\log2)/8$, 
	then $\|D\|=\|\log(e^{-s_1V}e^{t_1Z})\|\leq 
	-1/2 \log\left(2-e^{2\|t_1 
	Z\|+2\|s_1V\|}\right)
	<\pi$
	(see some of the Baker-Campbell-Hausdorff 
	series bounds in  
	\cite{beltita_Smooth_homogeneous_structures} 
	or 
	\cite{varadarajan_Lie_groups}). This implies 
	that 
	$ D= {K-\diag(s_1V)+\diag(t_1Z)}$ because 
	$e^D=e^{K-\diag(s_1V)+\diag(t_1Z)}$ 
	and both anti-Hermitian exponents have norm 
	less than $\pi$ (see for example 
	Corollary 4.2 iii) of 
	\cite{chiumiento_normal_logarithms}). But, 
	since $\diag(K)=0$ and $D\in\dah$, then 
	\begin{equation}
	\label{ec D es la suma de las diags de V y de 
	Z}
	D=-\diag(s_1V)+\diag(t_1Z) .
	\end{equation}
	and $K=0$.
	%
\end{proof}

\begin{teo}\label{teo: hopf rinow local} 
	(Local Hopf-Rinow theorem)
	There exists $\mathcal{W}_b\subset\ob=\ob^{\ua}$ a neighborhood (with 
	the 
	distance defined in \eqref{def: distancia rectificable}) of  $b\in 
	\D(\khh)$ with $b_{i,i}\neq b_{j,j}$ for $i\neq 
	j$,
	such that for 
	every $\rho\in \mathcal{W}_b$ there exists a short curve $\gamma$ in $\ob$ 
	that 
	joins $b$ with $\rho$,  and 
	$\gamma:[0,1]\to\mathcal{W}_b\subset\ob$, 
		$$
		\gamma(t)=e^{t(K_\rho+D_\rho)}be^{-t(K_\rho+D_\rho)},
		$$ 
	with $K_\rho\in\kah$, $D_\rho\in\dah$, $\|K_\rho\|, \|D_\rho\|, 
	\|K_\rho+D_\rho\| < 
	\frac{\log(2)}{4}$, and $(K_\rho+D_\rho)$ a minimal operator in $\kah+\dah$.

\end{teo} 
\begin{proof}
	
	If we consider the isotropy compact generalized flag manifold 	
	$\mathcal{P}=\mathcal{U}(\h)/\D {\left(\mathcal{U}(\h)\right)}$ then for 
	$\rho_0\in\mathcal{P}$ there exists a neighborhood 
	$$
	\mathcal{V}_{\rho_0}=\{L_u\rho_0 : u=e^X, \text{for } X\in\bah, 
	\|X\|<\pi/2\}
	$$ 
	where a local Hopf-Rinow theorem holds (see Theorem II-1 and 
	Example 1 of \cite{dmr2}). That is, for every $\rho\in\mathcal{V}_{\rho_0}$ 
	there exist a minimal operator $X\in\bah$ with $\|X\| {<}  
	\pi/2$ and a minimal uniparametric group curve $\gamma:[0,1]\to 
	\mathcal{P}$, 
	$\gamma(t)=L_{e^{tX}} \rho_{0}$ joining $\gamma(0)= \rho_{0}$ and 
	$\gamma(1)= \rho$. 
	
	The generalized flag manifold 
	$\mathcal{P}=\mathcal{U}(\h)/\D{\left(\mathcal{U}(\h)\right)}$ can be 
	identified with the unitary orbit of $b\in\D(\kh^h)$ with $b_{i,i}\neq 
	b_{j,j}$
		as well as its tangent spaces as we have done with $\ob$ in Remark 
	\ref{rem: espacio tangente Ob}:
	$$
	T_c \ \mathcal{O}_b^{\uh }\simeq 
	T_1 \ \uh/T_1 \ \D \left(\uh
	\right)=
	\bh^{ah}/\dbh^{ah}
	$$

	Since we are using the adjoint action $L$, then $L_{u}b=ubu^*$ in this 
	context. And if we consider $\rho_0=b$ we can conclude 
	that for any $\rho=e^Kbe^{-K}\in \left(\ob\cap \mathcal{V}_b\right)$ 
	with $K\in\kah$ 
	(see \eqref{eq: las 
	orbitas son iguales con distintos subgrupos}), there 
	exists a minimal operator $Z\in \bah$ with $\|Z\| < 
	\pi/2$, such that 
	\begin{equation}\label{eq: gama curva minimal en Ukd}
	\gamma:[0,1]\to \mathcal{O}_b^{\uh} ,  \gamma(t)= e^{tZ}b e^{-tZ}, \text{ 
	with } \gamma(0)=\rho_0= b \text{ and }  \gamma(1)= \rho= e^Kbe^{-K}= e^Z 
	be^{-Z}.
	\end{equation}
	Note that in this case $e^Z$ cannot be any element of $\uh$ 
	because 
	$e^Kbe^{-K}= e^Z be^{-Z}$ 
	implies that $e^Z=e^K e^D$ for $D\in\dah$, and therefore $e^Z\in \ukd$. 
Moreover, we will show that choosing a smaller neighborhood $Z$ can be written 
as $Z=K'+D'$, with 
$K'\in\kah$, $D'\in\dah$.
In order to obtain this last assertion recall from  Lemma 3.14 
of \cite{bottazzi_varela_studia} that there exists 
$\varepsilon_0>0$ such that if $u\in\ukd$ satisfies 
$\|u-1\|<\varepsilon_0$, then there exist $K'\in\kah$ and $D'\in(\dbh)^{ah}$ 
with $u=e^{K'+ D'}$ for $\|K'\|, \|D'\|, \|K'+D'\|<\frac{\log2}4$ (see 
Definition 2.1 and the proof of Lemma 3.14 of \cite{bottazzi_varela_studia}). 

Then define the neighborhood of $b$ in $\ob$:
\begin{equation*}
\begin{split}
\mathcal{W}_{b}&=\{u b u^*: 
u=e^{Z}\in \ukd, Z\in\bah, \|Z\|< \log(1+\varepsilon_0) \}
\end{split}
\end{equation*}
with $\varepsilon_0$ from Lemma 3.14 of \cite{bottazzi_varela_studia}.
 Note that $\mathcal{W}_{b}
\subsetneq (\mathcal{V}_b\cap \ob) 
\subsetneq 
\left(\mathcal{V}_b\cap  \mathcal{O}_b^{\uh}\right)$.
It is apparent that if $u=e^{Z}\in \ukd$, with $Z\in\bah, \|Z\|< 
\log(1+\varepsilon_0)$, then $\|e^Z-1\|\leq e^{\|Z\|}-1<\varepsilon_0$.
Applying the mentioned lemma, this implies that in this case there exist 
$K'\in\kah$ and $D'\in\dbh^{ah}$ 
with $u=e^{K'+ D'}$ for $\|K'\|, \|D'\|, \|K'+D'\|<\frac{\log2}4$. 
Following the discussion after \eqref{eq: gama curva minimal en Ukd}, if 
$\gamma:[0,1]\to \ob^{\uh}$ is the short curve $\gamma(t)=e^{t 
Z}be^{-tZ}\subset \mathcal{V}_b$ for $Z\in\bah$ a minimal operator such that 
$\gamma(1)=e^K 
be^{-K}\in \ob$ (for $K\in\kah$), then it must be $e^Z=e^K e^D\in\ukd$. 
Moreover, since 
$\|e^Z-1\|<\varepsilon_0$, there exist $K'\in\kah$ and $D'\in\dbh^{ah}$ 
satisfying 
$$
e^Z=e^{K'+ D'}  \text{, for } \|K'+D'\|<\frac{\log2}4 \Longrightarrow Z=K'+D'
$$ 
because $Z$ and $K'+D'$ have norm smaller than $\pi$.
Then the entire curve $\gamma:[0,1]\to\ob^{\uh}$ 
 is included 
 in 
 $\ob=\ob^{\ua}$. In this case, being $Z=K'+D'$, the distance from $\rho_0=b$ 
 to $\rho=e^Z 
 be^{-Z}$ is the same either if we consider the Finsler metrics in $\ob$ or in 
 $\ob^{\uh}$ (see \eqref{eq: igualdades entre infimos para calculo normas 
 cociente}). Then $\gamma(t)=e^{tZ} be^{-tZ}=e^{t(K'+D')}be^{-t(K'+D')}$, 
 $\gamma:[0,1]\to \ob$
 defines a short curve between $b$ and 
 $\rho=e^Kbe^{-K}=e^{(K'+D')}be^{-(K'+D')}$, with $K'+D'=Z$ a minimal operator 
 of $\kah+\dah$. The statement of the theorem follows after substituting 
 $K'=K_\rho$ and $D'=D_\rho$.

The element 
	$\rho=e^{K}be^{-K}\in\mathcal{W}_b$ was chosen arbitrarily,  
	so we have proved that $\mathcal{W}_b$ is a geodesic neighborhood of $b$ in 
	$\ob$.
\end{proof}
\begin{rem}
 Observe  that the unitary $e^{K_\rho+D_\rho}\in \ukd$ mentioned in the 
 previous 
 theorem might not belong to $\kk+\C$, but 
 $\gamma(t)=e^{t(K_\rho+D_\rho)}be^{-t(K_\rho+D_\rho)}\in \ob$ for every 
 $t$ (see Remark \ref{rem: orbitas y metrica finlser iguales}).
\end{rem}
\begin{rem}
	Let $c=e^{K_0}be^{-K_0}\in\ob$, with $K_0\in\kah$. The action 
	$L_u(c)=ucu^*$, for $u\in\uu_{\kk+\C}$ is invariant for the distance 
	defined 
	in $\ob$ and therefore the previous result also holds in this 
	case. 
	That is, there exists a geodesic neighborhood $\mathcal{W}_c$ of $c$ such 
	that every $\rho\in\mathcal{W}_c$ is joined with $c$ by short curves 
	included in $\ob$ of the form 
	$L_{e^{K_0}}\circ\gamma$ (for $\gamma$ the curve described in Theorem 
	\ref{teo: hopf rinow local}).
\end{rem}

Here we recall some results and the notation used in \cite{dmr1} and state its 
translation to the particular example we are studying. In the work mentioned, 
the minimality of a curve $\gamma:[0,\frac{\pi}{2\|Z\|}]\to \ob$,  $\gamma(t)=e^{tZ}be^{-tZ}$, with $Z$ be a minimal lifting of $x\in T_b\ob$, was proved using a unitary reflection $r_0$ in a Hilbert space with certain properties (see Definition 2.4 in \cite{dmr1}), a 
representation of 
$\mathcal{A}$ in $B(\h)$ with particular properties (in our case the identity 
representation verifies them) and a map $F:\ob\to \text{Gr}(\h)$ ($\text{Gr}(\h)$ is the Grassmann manifold of $\h$) defined by 
$F(u b u^*)=ur_0u^*$.
The unitary reflection $r_0$ used there is

\begin{equation}\label{defi ro general}
r_0(x)=\left\{
 		\begin{array}{rcl} 
 		x , \text{ if } x\in S_b\\
 		-x, \text{ if } x\in S_b^\perp,
 		\end{array} \right. 
\end{equation}

 where $S_b$ is the closure of $\Omega=\{x\in \h: x= U \xi, \text{ 
 			for } U \text{ a diagonal in } 
 			\ua  \}$ and $\xi\in\h$ 
 			certain 
 			vector satisfying Definition 
 			4.1 of \cite{dmr1}.
 		
Now, if we consider the particular case in 
which $Z$ is a minimal operator such 
that $\|Z\|=\|c_{j_0}(Z)\|$, $c_{j_0}(Z)_{j_0}=Z_{j_0,j_0} = 0$ and 
$c_{j_0}(Z)_{j}=Z_{j,j_0} \neq 0$ for all $j\neq j_0$ 
(see the example of \eqref{defi Z2} and Lemma \ref{lem realiza norma en columna 
entonces col ortog}) we can 
be much more specific about $r_0$ and 
$\xi$.

After the corresponding translation to 
this case 
\begin{equation}
\label{def xi eta r0}
\xi=i\, e_{j_0}, \ S_b=\text{gen}\{\xi\}\ \text{ and }
 r_0(x)=\left\{
 \begin{array}{rcl} 
 x &,& \text{ if } x\in \text{gen}\{\xi\}=\text{gen}\{e_{j_0}\}\\
 -x&,& \text{ if } x\in  \text{gen}\{\xi\}^\perp =\text{gen}\{e_{j_0}\}^\perp
 \end{array} \right. .
\end{equation}

Moreover, $\xi$ fulfills  Definition 4.1 
in \cite{dmr1}, since 
$Z^2\xi=-\|Z\|^2\xi$, $r_0(\xi)=\xi$ 
and $r_0(Z\xi)=-Z\xi$. Therefore, 
$\gamma(t)=e^{tZ}be^{-tZ}$ minimizes 
length between the points $\g(0)=b$ 
and $\g(t)$ if $0\leq t\leq 
\frac{\pi}{2\|Z\|}$.


In this context, the map $F_\xi:\ob\to\mathscr{S}\subset\h$, $F_\xi(ubu^*)=ur_0u^*(\xi)$, where $\mathscr{S}$ is the unit sphere of $\h$, reduces length. That is,
if $\delta:[0,t_0]\to\ob$  and $v:[0,t_0]\to \mathscr{S}$, $v(t)=F_\xi(\delta(t))$
then (see Corollary 3.4 in \cite{dmr1})
$$
\ell(v)\leq L(\delta)
$$
with $L$ defined in \eqref{def: longitud rectificable} and $\ell$ the length in $\mathscr{S}$.

The following result is an application of Theorem 4.1 and Lemma 4.2 in 
\cite{dmr1} to our context.

\begin{prop} \label{coro coco curvas minimales en H}
 Let $Z$ and $V$ 
be minimal operators of 
$\left(\kh+\dd(\h)\right)^{ah}$ and 
consider the following curves in $\ob$ 
(defined in \eqref{def: orbita ob})
$$
\g(t)=e^{tZ} b e^{-tZ}\ \text{and}\ 
\delta(t)=e^{tV} b e^{-tV},\ t>0.
$$
Suppose additionally that there exists $0\leq t_0\leq\min\left\lbrace  \frac{\pi}{2\|Z_2\|}; \frac{\pi}{2\|V\|}\right\rbrace $ such that $\g(t_0)=\delta(t_0)$. 

Then, following the previous notation, 
\be
\item $w(t)=F_\xi(\g(t))$ and 
$v(t)=F_\xi(\delta(t))$ both 
are geodesics in the sphere 
$\mathscr{S}\subset \h$ and 
$$\ell(w)=L(\g)=L(\delta)=\ell(v).$$
\item $\g $ and $\delta $ minimize length between the points $b$ and $\g(t)$ and $\delta(t)$ respectively, if $0\leq t\leq t_0$.
\item $v(t)=e^{tV} r_0 e^{-tV}(\xi)=e^{tZ} r_0 
e^{-tZ}(\xi)=w(t)$, for $0\leq t\leq t_0$ and $r_0$ the unitary reflection 
defined in \eqref{defi ro general}.
\ee 

\begin{proof}
As it was mentioned before, items (1) and (2) are 
a direct consequence of Theorem 
4.1 and Lemma 4.2 in \cite{dmr1}. If $\g(t_0)=\delta(t_0)$, $w(t_0)$ and 
$v(t_0)$ match. Since geodesics for fixed ending points in the unit sphere 
$\mathscr{S}$ are 
unique (maximum circles), then $w(t)=v(t)$, for all $0\leq t\leq t_0$.
\end{proof}
\end{prop}

Observe that the assumption of existence of 
such $t_0$ is possible even if 
$\g(t_1)=\delta(t_2)$ for $t_1\neq t_2$, since 
$\delta(t)$ can be re-scaled defining 
$\tilde{\delta}(t)=e^{t\frac{t_2}{t_1}V}be^{-t\frac {t_2}{t_1}V}$, for $t\in 
\left[0,\frac{\pi}{2\frac{t_2}{t_1}\|V\|}\right]$and then 
$\tilde{\delta}(t_1)=\g(t_1)$.

\begin{lem} \label{lem: equal columns} 
	Let $b\in\mathcal{D}(\kk)^h$ with $b_{i,i}\neq b_{j,j} $  for $i\neq j$, 
$Z, V \in(\kk(\h)+\dd(\bh))^{ah}$ be minimal operators such that  
for some $j_0\in \N$ $\|Z\|=\|c_{j_0}(Z)\|$, $c_{j_0}(Z)_{j_0}=Z_{j_0,j_0} = 0$ 
and 
$c_{j_0}(Z)_{j}=Z_{j,j_0} \neq 0$ for all $j\neq j_0$. Moreover, if there 
exists $t_0$ with $0< t_0\leq  \frac{\pi}{2\|Z\|}$ that 
satisfies 
$e^{t_0 Z}be^{-t_0 Z}=e^{s_0 
	V}be^{-s_0 V}$, for $s_0\in\left[0, \frac{\pi}{2\|V\|}\right]$, then
$$
s_0 c_{j_0}(V)=t_0 c_{j_0}(Z),
$$
that is, the  $j_0^{th}$ column of $Z$ is a multiple of the $j_0^{th}$ column of $V$.
\end{lem}

\begin{proof}
We can suppose that $j_0=1$.  
 First, consider $W=\frac{s_0}{t_0}V$,
 $$
 \g(t)=e^{tZ} b e^{-tZ}\ \text{and}\ \delta_0(t)=e^{tW} b e^{-tW}
 $$
 for $t\in\left[0,\frac{\pi}{2 \|Z\|}\right]$.
 Then $\gamma(t_0)=e^{s_0 
 	V}be^{-s_0 V}=e^{t_0 s_0/t_0 
 	V}be^{-t_0 s_0/t_0 V}=\delta_0(t_0)$. Moreover, since also $W$ is 
 a minimal operator, then $\gamma$ and $\delta_0$ are short curves in 
 the interval $[0,t_0]$ and then the length of 
 $\gamma\left|_{[0,t_0]}\right.$ 
 equals that of $\delta_0\left|_{[0,t_0]}\right.$. That is,	
 $t_0\|Z\|=t_0\|W\|$ (see for example Theorem 4.1 of \cite{dmr1}) and 
 therefore in 
 particular
 \begin{equation}\label{eq: norma Z igual norma W}
 \|Z\|=\|W\|.
 \end{equation}
 
 Using the preceding notations of Proposition 
 \ref{coro coco curvas minimales en H}, the assumptions made here imply that 
 $v(t)=F_\xi(\g(t))=F_\xi(\delta_0(t))=w(t)$, for 
 $t\in\left[0,t_0\right]$. Then their 
 derivatives 
 coincide for 
 every  $t\in[0,t_0]$
 \begin{equation}
 \label{eq: igualdad de derivadas de v y de w}
 w'(t)=\left( W e^{tW} r_0 e^{-tW}-e^{tW} r_0 e^{-tW} 
 W\right)(\xi)=\left(Ze^{tZ} r_0 e^{-tZ}-e^{tZ} r_0 e^{-tZ} 
 Z\right)(\xi)=v'(t).
 \end{equation}
 where $ \xi=i e_1$, $\eta=\frac{c_1(Z)}{\|c_1(Z)\|}$ and the reflection $r_0$ 
 are
 defined in \eqref{def xi eta r0} and fulfill $r_0(\xi)=\xi$ and 
 $r_0(\eta)=-\eta$.
 
 Then, if  we evaluate \eqref{eq: igualdad de derivadas de v y de w} in 
 $t=0$ 
 $$
 w'(0)=\left(Wr_0-r_0W\right)(\xi)=\left(Zr_0-r_0Z\right)(\xi)=v'(0).
 $$
 Hence, since $r_0(\xi)=\xi$ and 
 $r_0(c_1(Z))=r_0(\|c_1(Z)\|\eta)=\|c_1(Z)\|\, 
 r_0(\eta)=-\|c_1(Z)\|\eta=-c_1(Z)$,
 \begin{equation}\label{eq: relacion columnas}
 \begin{split}
 Wr_0(\xi)-r_0W(\xi)&=Zr_0(\xi)-r_0Z(\xi)\\
 W(ie_1)-r_0W(ie_1)&=Z(ie_1)-r_0Z(ie_1)\\
 ic_1(W)-ir_0(c_1(W))&=ic_1(Z)+r_0\left(ic_1(Z)\right)\\
 i (I-r_0) \left(c_1(W)\right)&=i2c_1(Z)
 \end{split}
 \end{equation}
 On the other hand, if we consider the decomposition  
 $\h=\text{gen}\{\xi\}\oplus (\text{gen}\{\xi\})^\perp$ then the identity 
 operator $I$ 
 and $r_0$ can be 
 matricially described as 
 $$
 I=\begin{pmatrix}
 1&0\\
 0&1
 \end{pmatrix}\ \text{and}\ r_0=\begin{pmatrix}
 1&0\\
 0&-1
 \end{pmatrix},$$
 respectively. 
 Then, \eqref{eq: relacion columnas} implies that  
 $(I-r_0)(c_1(W))=2 (c_1(W)-W_{1,1} e_1)= 2 c_1(Z)
 $ and then
 \begin{equation}
 \label{eq relacion cols V y Z2}
 c_1(W)-W_{1,1} e_1=   c_1(Z)
 \end{equation}
 and  $\|c_1(Z)\|=\|Z\|=\|W\|$ (see \eqref{eq: norma Z igual norma W}). 
 This implies that $\|c_1(W)-W_{1,1} 
 e_1\|=\|W\|$, and therefore
 $$
 W_{1,1}=0,
 $$ 
 since otherwise $\|c_1(W)\|>\|W\|$, which is a contradiction.
 Therefore, returning to \eqref{eq relacion cols V y Z2} we obtain that
 $c_1(W)=c_1(Z)$ that implies that
 $$
 c_1\left(\frac{s_0}{t_0} V\right)=c_1(Z)
 $$
 which ends the proof.
\end{proof}

Next, we obtain the second main result of this section.

\begin{teo} \label{teo: unicidad curva minimal}
	Let $b\in\mathcal{D}(\kk)^h$ with $b_{i,i}\neq b_{j,j} $  for $i\neq j$, 
$Z\in(\kk(\h)+\dd(\bh))^{ah}$ be a minimal operator such that  
for some $j_0\in \N$, holds that $\|Z\|=\|c_{j_0}(Z)\|$, 
$c_{j_0}(Z)_{j_0}=Z_{j_0,j_0} = 0$,
$c_{j_0}(Z)_{j}=Z_{j,j_0} \neq 0$ for all $j\neq j_0$, the sequence 
$\{\diag(Z)_{j,j}\}_{j\in\N}$ has more that one not null accumulation points, 
and 
$\g(t)=e^{tZ}be^{-tZ}$, for $t\in \left(0, 	\frac{\log 2}{8 \|Z\|}\right)$.

	Then there is not any minimal operator $V\in  
	(\kk+\C)^{ah}$ such that $\delta(t)=e^{tV} b e^{-tV}$, 
	$t\in\left[0,\frac{\pi}{2\|V\|}\right]$ satisfies $\g(t_0)=\delta(s_0)$ for 
	$t_0, s_0$ in the respective domains.
\end{teo}

\begin{proof} As done before, we are going to prove only the case $j_0=1$. 
	Suppose that there exists a minimal operator $V\in  (\kk+\C)^{ah}$ such 
	that $\delta(s_0)=e^{s_0V} b e^{-s_0V}=\g(t_0)$, for 
	$s_0\in\left(0,\frac{\pi}{2\|V\|}\right]$ and   
$t_0\in \left(0,   \frac{\log 2}{8 \|Z\|}\right)$. 
Note that in particular $t_0<\frac{\log 2}{8 \|Z\|}< \frac\pi{2\|Z\|}$ which 
implies that $\gamma$ is 
a short curve in all its domain.
	
	Applying Lemma \ref{lem: gamma y delta se cruzan entonces hay una 
	diagonal...}, $-\diag(s_0V)+\diag(t_0Z)\in \dah$ is such that 
	\begin{equation}
	\begin{split}
	\label{eq: la diag D es la resta de las diags de V y Z2}
	e^{-s_0V}e^{t_0Z}&=e^{-\diag(s_0V)+\diag(t_0Z)}\\
	\text{ and } \|t_0 Z\|&=\|s_0 V\|.
	\end{split}
	\end{equation}
	Then, the fact that $Z$ satisfies $\|Z\|=\|c_1(Z)\|$ (see 
	\eqref{ecuacion que cumple D0 minimal de Zo}) implies that $V$ 
	must fulfill that
	\begin{equation}
	\label{igualdad entre norma V y con Z2}
	\|V\|=\frac{t_0}{s_0}\|c_1(Z)\|
	\end{equation}

	Using Lemma \ref{lem: equal columns}  follows that $t_0 c_1(Z)=s_0 
	c_1(V)$,  
	and therefore \eqref{igualdad entre norma V y con Z2} implies that 
	$\|V\|=\frac{t_0}{s_0}\|c_1(Z)\|=\frac{t_0}{s_0}\left\|\frac{s_0}{t_0}c_1(V)\right\|
	=\|c_1(V)\|$. 
	
	Then Lemma \ref{lem realiza norma en columna entonces col ortog} implies 
	that 
	$c_1(V)$ is orthogonal to every 
	other column of $V$. This property also holds for $c_1(Z)$ and the 
	columns of 
	$Z$. 
	Recall the notation $\xi=i e_1$ and 
	$\eta= \frac{c_1(Z)}{\|c_1(Z)|\|}= 
	\frac{\frac{s_0}{t_0}c_1(V)}{\|\frac{s_0}{t_0}c_1(V)\|}= 
	\frac{c_1(V)}{\|c_1(V)\|}$, and 
	consider
	$$
	\xi + \eta=i e_1+\frac{c_1(Z)}{\|c_1(Z)\|}=i 
	e_1+\frac{c_1(V)}{\|c_1(V)\|}  \in \h .
	$$ 
	A direct computation shows that $\xi+\eta$ is an eigenvector of $Z$ and 
	$V$ of the 
	eigenvalue 
	$i\|Z\|=i\|c_1(Z)\|=i\frac{s_0}{t_0}\|c_1(V)\|=i\frac{s_0}{t_0}\|V\|$ 
	(see 
	for 
	example the proof of Theorem 2 in \cite{bottazzi_varela_DGA} for the 
	self-adjoint case). The previous comments imply 
	that 
	$$
	e^D (\xi+\eta) = e^{-s_0V}e^{t_0Z}   (\xi+\eta) =
	e^{-i\frac{s_0^2}{t_0}\|V\|} e^{i t_0\|Z\|}  (\xi+\eta) = e^{i 
		(t_0-s_0)\|Z\|}  (\xi+\eta) 
	$$
	where in the last equality we used $\|V\|=\frac{t_0}{s_0}\|Z\|$ and the series expansion of the exponentials.
	Then, using \eqref{eq: la diag D es la resta de las diags de V y Z2} we 
	write 
	\begin{equation}
	\label{exponenciales tceros y sceros}
	e^D(\xi+\eta)=e^{-\diag(s_0 V)+\diag(t_0 Z)}(\xi+\eta)=e^{i 
		(t_0-s_0)\|Z\|}  (\xi+\eta).
	\end{equation}
	Therefore, considering the equality in each entry of \eqref{exponenciales 
		tceros y sceros}  we obtain 
	$$
	e^{(-s_0 V+t_0 Z)_{j,j}} (\xi+\eta)_{j}=e^{i (t_0-s_0)\|Z\|}  
	(\xi+\eta)_{j}
	$$ 
	for all $j\in \N$. The fact that $(\xi+\eta)_{j}\neq 0$ for all $j\in 
	\mathbb{N}$ implies that 
	$e^{(-s_0 V+t_0 Z)_{j,j}}=e^{i (t_0-s_0)\|Z\|}$ for every $j$. Since we 
	are 
	supposing that $t_0\in \left(0,  	\frac{\log 2}{8 \|Z\|} \right]$, then 
	the 
	exponent 
	$(-s_0 V+t_0 Z)_{j,j}$ is small enough and we can 
	conclude that 
	$$
	(-s_0 V+t_0 Z)_{j,j}=-s_0 V_{j,j}+t_0 (Z)_{j,j}=i (t_0-s_0)\|Z\|
	$$
	for all $j\in \N$. But this is a contradiction 
	since we are supposing $\diag(V)=d+i\theta I$, with $d\in \kah$, 
	$\theta\in\R$, 
	and we know that $\diag(Z)$ has more than one (not null) limit. 
	Therefore, a 
	minimal operator $V\in (\aaa)^{ah}$ cannot form a curve $\delta(t)=e^{tV} b 
	e^{-tV}$ that crosses $\g$ for $t>0$ in a certain small enough neighborhood 
	of 
	$b$.
\end{proof}

\begin{cor}
	If we consider $Z_2$ as defined in \eqref{defi Z2} for every neighborhood 
	$\mathcal{X}_b$ of $b$ in $\ob$ there 
	exist elements  $(e^{t_0Z_2}be^{-t_0Z_2})\in\mathcal{X}_b$ such that there 
	is 	not any short curve of the form $\delta(t)=e^{tV}be^{-tV}$ with $V\in 
	(\kk+\C)^{ah}$ 	that joins $b$ with $e^{t_0Z_2}be^{-t_0Z_2}$. In fact, this 
	is true for  $e^{tZ_2}be^{-tZ_2}$, for every $t$ in certain interval.
\end{cor}

\begin{proof} Observe that the operator $Z_2$ satisfies every assumption needed 
by the operator $Z$ in Theorem \ref{teo: unicidad curva minimal}.Then the proof 
is a 
direct application of the previous theorem.
\end{proof}

\begin{rem}
	Note that the situation mentioned in the previous corollary applies to the 
	geodesic neighborhood $\mathcal{W}_b$ obtained in Theorem \ref{teo: hopf 
		rinow local} even when in that case every element of $\mathcal{W}_b$ is 
	reached by a short curve.
\end{rem}

\section{Appendix}
In this section we include various results concerning minimal anti-Hermitian 
operators in $\bah$.

\begin{lem} \label{lem realiza norma en columna entonces col ortog}
For a given fixed orthonormal basis of 
$\h$, let $V\in \bah$ be such that there 
exists $j_0
	\in \N$ that satisfies $\|V\|=\|c_{j_0}(V)\|$ (where $c_j(V)$ is the 
	$j^{\text{th}}$ column of the corresponding matrix representation of $V$ in 
	the fixed basis). Then   
	\begin{equation}
	\label{eq: cj0V perp cjV}
		c_{j_0}(V)\perp c_j(V) , \ \ 
		\forall 
	j\neq j_0.
	\end{equation}
	If 
	$\left(c_{j_0}(V)\right)_{j_0}=V_{j_0,j_0}=0$
	 then $V$ is a minimal operator.
	
	Moreover, if  
	$c_{j_0}(V)_{j}=V_{j,j_0}\neq0$ for 
	all $j\neq 
	j_0$, then $V$
	has a unique 
	minimizing diagonal defined by
		\begin{equation}
		\label{eq: explicitacion de la 
		unica diagonal de V minimal}
		V_{j,j}=-\frac{\left\langle c_j(V)_{\widecheck
				j},c_{j_0}(V)_{\widecheck j}\right\rangle }{V_{j,j_0} }, 
				\text{ for } j\neq j_0
 		\end{equation}
		\\
		where $c_k(X)_{\widecheck l}\in 
		\h\ominus 
		\text{gen}\{e_l\}$ is the  
		element obtained  after taking off 
		the $l^{\text{th}}$ entry of 
		$c_k(X)\in\h$.
\end{lem}
\begin{proof}
	Note that we can suppose that $j_0=1$ to simplify the notation. Similar 
	considerations could be done for the $j_0^{th}$ column. 
	
	In the matrix representation corresponding to the fixed orthonormal basis $\{e_j\}_{j\in \N}$, we can consider  
	$$
	x= \cos(t) e_1+\sin(t) e_j , \text{ for } j\neq 1.
	$$
	Observe that $\|x\|=1$, and then it must hold $\|Vx\|\leq \|V\|$. 
	Let us consider $f:\R\to \R$ such that 
	\begin{equation}\label{ec definicion f(t)} 
	\begin{split}
	f(t)&=\|V(\cos(t) e_1+\sin(t) e_j)\|^2=\|\cos(t) V(e_1)+\sin(t) V(e_j)\|^2\\
	&= \| \cos(t)  c_1(V)+\sin(t)   c_j(V)\|^2
	\\ &=\langle\cos(t)  c_1(V)+\sin(t)   c_j(V) , 
	\cos(t)  c_1(V)+\sin(t)   c_j(V) \rangle \\
	&=\cos^2(t) \| c_1(V)\|^2+\sin^2(t) \| c_j(V)\|^2+ 
	2 \cos(t)  \sin(t) \  \text{Re}\langle c_1(V),c_j(V) \rangle .\\
	\end{split}
	\end{equation}
	Then
	\begin{equation*}
	\begin{split}
	f'(t) =&-2\sin(t) \cos (t) \| c_1(V)\|^2+2\sin(t) \cos (t) \| c_j(V)\|^2+ 
	\\ &+2 \left(\cos(t)^2-  \sin(t)^2\right) \  \text{Re}\langle c_1(V),c_j(V) \rangle \\
	=& \sin(2t)\left( \|c_j\|^2-\|c_1\|^2\right)+ 2 \cos(2t) \  \text{Re}\langle c_1(V),c_j(V) \rangle.
	\end{split}
	\end{equation*}
	
	Then, if $\text{Re}\langle c_1(V),c_j(V) \rangle>0$
	$$
	f'(0)=2 \text{Re}\langle c_1(V),c_j(V) \rangle >0
	\ \text{ and } \ f(0)= \|c_1(V)\|^2
	$$
	and then $f'(t_1)>0$ for some $t_1>0$, which implies that $f(t_1)>\|c_1(V)\|^2$, a contradiction.
	
	On the other hand, if $\text{Re}\langle c_1(V),c_j(V) \rangle<0$
	$$
	f'(0)=2 \text{Re}\langle c_1(V),c_j(V) \rangle <0
	\ \text{ and } \ f(0)= \|c_1(V)\|^2
	$$
	and then $f'(t_2)< 0$ for some $t_2<0$, which implies that $f(t_2)>\|c_1(V)\|^2$, a contradiction.
	
	Therefore it must be $\text{Re}\langle c_1(V),c_j(V) \rangle =0$.
	
	Now consider $z= \cos(t) e_1+ i 
	\sin(t) e_j\in\h$, that also satisfies 
	$\|z\|=1$. Then following the steps we 
	used in the case of $x =\cos(t) e_1+ 
	\sin(t) e_j$ but using $z$, it can be 
	proved that $0=\text{Re}(-i)\langle 
	c_1(V),c_j(V) \rangle 
	=\text{Im}\langle c_1(V),c_j(V) 
	\rangle$.

	In order to prove the last part of the lemma observe that 
		\begin{enumerate}
			\item as proved in the first part of this lemma, $c_{1}(V)\perp 
			c_{j}(V)$ for all $j\neq 1$,
			\item and the assumptions 
			\begin{enumerate}
				\item $c_{1}(V)_{1}=V_{1,1}=0$,  
				\item $c_{1}(V)_{j}=V_{j,1}\neq 0$ for $j\neq 1$ 
				\item and the equality 		$\|V\|=c_{1}(V)$ 
			\end{enumerate}
		\end{enumerate} 
	Then the proof of the minimality of $V$ follows applying Theorem 
	2.2 from \cite{andruchow_mata_recht_mendoza_varela_LAA} substituting 
	$\mathcal{A}$ with $\bh$, $\mathcal{B}$ with $\dd(\bh)$, $\rho$ 
	with the identity, $\xi$ with $i \, e_{1}$ and $Z$ with $V$.
	Note 
	that we only need assumptions (1), (2)(a) and (2)(c) to prove that 
	$$
	V^2 (i\, e_{1})=-\|V\|^2 i\, e_{1} \text{ and that } \langle V (i\, 
	e_{1}), D(i\, e_{1})\rangle=\langle i \, c_{1}(V), i\, 
	D_{1,1}\rangle=0
	$$
	in order to fulfill the assumptions of Theorem 
	2.2 from \cite{andruchow_mata_recht_mendoza_varela_LAA}.
	
	The equality \eqref{eq: explicitacion de la unica diagonal de 
	V minimal} follows after the condition $c_{1}(V)\perp c_j(V)$ for $j\neq 
	1$ and the fact that $c_{1}(V)_j=V_{j,1}\neq 0$ for those $j$. 
	
	Moreover, if we consider $V+D$, for $D\neq 0$, and $D_{1,1}\neq 0$, 
	follows that 
	$\|c_{1}(V+D)\|=\|c_{1}(V)+D_{1,1}e_{1}\|>\|c_{1}(V)\|=\|V\|$ 
	and therefore $V+D$ cannot be minimal. Now suppose $D_{1,1}= 0$.
	Direct computations show that 
	\begin{equation}
	\begin{split}
	\left\|(V+D) 	\frac{c_{1}(V)}{\|c_{1}(V)\|}\right \|
	&=\frac1{\|c_{1}(V)\|}\|V c_{1}(V) 
	+ D c_{1}(V)\|=\frac1{\|c_{1}(V)\|} \left\|-\|c_{1}(V)\|^2 e_{1}+ 
	D  c_{1}(V)\right\|
	\\
	&=\left\|-\|c_{1}(V)\| 
	e_{1}+\frac1{\|c_{1}(V)\|}  D  c_{1}(V)\right\|> 
	\|c_{1}(V)\|=\|V\|. 
	\end{split}
	\end{equation}
	In the previous strict inequality we have used (2)(a), (2)(b), $D\neq 0$ 
	and $D_{1,1}=0$.

Then $\|V+D\|>\|V\|$ for $D\neq 0$, which implies that the diagonal defined in 
\eqref{eq: explicitacion de la unica diagonal de V minimal} is the only possible
minimizing 
	diagonal of $V$.
\end{proof}
Another way to prove equation 
\eqref{eq: cj0V perp cjV} of the first 
part of the previous Lemma \ref{lem realiza norma en columna entonces col 
ortog} is using Corollary \ref{coro del teo sain} of the following 
theorem.
\begin{teo}[Sain, \cite{sain}]\label{sain}
	Let $\h_1$, $\h_2$ be Hilbert spaces and $T\in \mathcal{B}(\h_1,\h_2)$. Given
	any $x\in\h_1$, $\|Tx\|=\|T\|$ if and only if the following two conditions are
	satisfied:
	\begin{enumerate}[label=\roman*)]
		\item $\left\langle x, y \right\rangle = 0$ implies that $\left\langle  Tx, Ty \right\rangle = 0$,
		\item $\sup\{\|Ty\| : \|y\| = 1, \left\langle  x, y \right\rangle = 0\} \leq \|Tx\|$.
	\end{enumerate}
\end{teo}

\begin{cor}\label{coro del teo sain}
	Consider $\h=\h_1=\h_2$ and $V\in \bah$. Then there exists $j_0\in \N$ such 
	that $\|V\|=\|V(e_{j_0})\|=\|c_{j_0}(V)\|$, if and only if
	\begin{enumerate}[label=\roman*)]
		\item  $\left\langle e_{j_0},e_j 
		\right\rangle=0$ implies that $\left\langle Ve_{j_0},Ve_j 
		\right\rangle=\left\langle 
		c_{j_0}(V),c_{j}(V) \right\rangle = 0$ for each $j\in \N$, $j\neq j_0$, 
			\item $\sup\{\|c_j(V)\| : j\in \N\} \leq \|c_{j_0}(V)\|$.
	\end{enumerate}
\end{cor}
\begin{proof}
	The proof is a direct consequence of 
	Theorem \ref{sain} after observing 
	that 
in item i) of the corollary is equivalent to say that
$\left\langle e_{j_0},y\right\rangle=0$ implies that 
 $\left\langle Ve_{j_0},Vy\right\rangle=0$ for any $y\in\h$. 
\end{proof}

\bibliographystyle{abbrv}

\bibliography{nuestra_bibliografia}
\end{document}